\documentclass{article}

\newcommand{\refeqn}[1] {Eq.~(\ref{#1})}
\newcommand{\reffig}[1] {Fig.~(\ref{#1})}

\usepackage{amsmath}
\usepackage{amssymb}
\usepackage{graphicx}
\usepackage{ucs}
\usepackage[utf8x]{inputenc}
\usepackage{amsthm}
\usepackage{subcaption}
\usepackage{tikz}
\usepackage[outline]{contour}

\usepackage[T1]{fontenc}
\usepackage{amsmath,amssymb}
\usepackage{authblk}
\usepackage{booktabs}

\newtheorem{proposition}{Proposition}

\usepackage{url}

\renewcommand{\vec}{\boldsymbol}

\begin{document}


\title{\bf{Lie Symmetry Analysis for Cosserat Rods}}

\author{
\small{\bf{Dominik~L.~Michels}}\\
\small{High Fidelity Algorithmics Group, Computer Science Department, Stanford University, 353 Serra Mall, MC 9515, Stanford, CA 94305}\\
\small{\texttt{michels@cs.stanford.edu}}

\and

\small{\bf{Dmitry~A.~Lyakhov}}\\
\small{Radiation Gaseous Dynamics Lab, A.~V.~Luikov Heat and Mass Transfer Institute, National Academy of Sciences of Belarus, P.~Brovka St 15, 220072 Minsk, Belarus}\\
\small{\texttt{lyakhovda@bsu.by}}

\and

\small{\bf{Vladimir~P.~Gerdt}}\\
\small{Group of Algebraic and Quantum Computations, Joint Institute for Nuclear Research, Joliot-Curie 6, 141980 Dubna, Moscow Region, Russia}\\
\small{\texttt{gerdt@jinr.ru}}

\and

\small{\bf{Gerrit~A.~Sobottka}}, \small{\bf{Andreas~G.~Weber}}\\
\small{Multimedia, Simulation and Virtual Reality Group, Institute of Computer Science II, University of Bonn, Friedrich-Ebert-Allee 144, 53113 Bonn, Germany}\\
\small{\texttt{\{sobottka,weber\}@cs.uni-bonn.de}}

}

\date{}



\maketitle

\begin{abstract}
We consider a subsystem of the Special Cosserat Theory of Rods and construct an explicit form of its solution that depends on three arbitrary functions in $(s,t)$ and three arbitrary functions in $t$. Assuming analyticity of the arbitrary functions in a domain under consideration, we prove that the obtained solution is analytic and general.
The Special Cosserat Theory of Rods describes the dynamic equilibrium of 1-dimensional continua, i.e. slender structures like fibers, by means of a system of partial differential equations.
\end{abstract}

\section{Introduction}
The Lie symmetry analysis of differential equations has become a powerful and universal approach to obtain group-invariant solutions and to perform their classification (cf.~\cite{Olivery10} and references  therein). Sophus Lie himself considered groups of point and contact transformations to integrate systems of partial differential equations (PDEs). His key idea was to obtain first infinitesimal generators of one-parameter symmetry subgroups and then to construct the full symmetry group. The study of symmetries of differential equations allows one to gain insight into the structure of the problem they describe. In particular, the existence of Lie symmetries means that one can find a decomposition of the differential equation system into a transformed system of reduced order and a set of integrators that in turn can be applied to develop more efficient numerical integration schemes for the governing differential equations.

In our contribution we focus on an equation subsystem of the Special Cosserat Theory of Rods (cf. \cite{Antman95}), a system of coupled partial differential equations, that govern the spatiotemporal evolution of the physical process of deformation of an one-dimensional continuum, the Cosserat rod (e.g. a fiber). In paper~\cite{ZWF07} the Lie symmetry analysis was applied to study symmetric properties of DNA modelled as a super-long elastic round rod. As a result, nontrivial infinitesimal symmetries of the dynamical Hamiltonian equations of the rod were detected and the related conserved quantities were derived. We consider here another model of the rod and apply the Lie symmetries to the subsystem of the governing system of partial differential equations. With assistance of computer algebra the Lie symmetry approach allowed us to construct a closed form of general analytical solution to the subsystem under consideration. Our motivation to do this research is based on the fact that the deformation modes of a Cosserat rod like bending, twisting, shearing, and extension typically evolve on different time scales which renders the problem inherently stiff (cf.~\cite{CurtissHirschfelder52}) and demands for appropriate methods for the numerical treatment of the governing PDE system. Knowledge of its structural properties can directly lead to more efficient solution methods.

\subsection{Specific Contributions}
We use computer algebra systems (specifically MAPLE, which provides sophisticated packages for the analysis of Lie symmetries in PDEs) in order to find Lie symmetries for proper systems and define the conditions under which they exist. In this regard our specific contributions are as follows.
\begin{itemize}
 \item We study a subsystem of the Special Cosserat Theory of Rods (cf. \cite{Antman95}) of the form $$\partial_t\vec{\kappa}(s,t)=\partial_s\vec{\omega}(s,t) + \vec{\omega}(s,t) \times \vec{\kappa}(s,t),$$by performing a Lie group analysis.
 \item We construct an explicit form of the solution of the subsystem that depends on three arbitrary functions in $(s,t)$ and three arbitrary functions in $t$.
 \item We prove that the obtained solution is analytic and general.
\end{itemize}


\section{Special Cosserat Theory of Rods}
In this section we give a recap of the Special Cosserat Theory of Rods.
Fibers can approximately be considered as one-dimensional continua that undergo bending, twisting, shearing, and longitudinal dilation deformation.
Following \cite{Antman95}, we consider the Euclidian 3-space $\mathbb{E}^3$ to be the abtract 3-dimensional inner product space. Its elements are denoted by lower-case, boldface, italic symbols. Let $\mathbb{R}^3$ be the set of triples of real numbers. Its elements are denoted by lower-case, boldface, sans-serif letters.

The motion of a special Cosserat Rod is given by
\begin{equation}
\label{eqn:ConfigurationCosseratRod}
(s, t)\mapsto\left(\boldsymbol{r}(s,t),\boldsymbol{d}_1(s,t),\boldsymbol{d}_2(s,t)\right),
\end{equation}
where $\boldsymbol{r}(s,t)$ is the centerline of the rod. It is furnished with a set of so-called orthonormal directors $\{\boldsymbol{d}_1(s,t), \boldsymbol{d}_2(s,t), \boldsymbol{d}_3(s,t)\}$. $\{\boldsymbol{d}_k\}$ is a right-handed orthonormal basis in $\mathbb{E}^3$, with $\boldsymbol{d}_3:=\boldsymbol{d}_1\times\boldsymbol{d}_2$. The directors $\boldsymbol{d}_1$ and $\boldsymbol{d}_2$ span the cross-section plane, see \reffig{fig:fiberCurveTriad}. The deformation of the rod is obtained if its motion defined by \refeqn{eqn:ConfigurationCosseratRod} is related to some reference configuration $$\{\boldsymbol{r}^{\circ}(s,t),\boldsymbol{d}^{\circ}_1(s,t),\boldsymbol{d}^{\circ}_2(s,t)\}.$$

\contourlength{1.6pt}
\newcommand{\renderScale}{0.15}
\newcommand{\outlineScale}{1.25*\renderScale}
\newcommand{\pxToMetric}{0.0352777778}
\newcommand{\horizontalCrop}{0.3}
\newcommand{\imageWidth}{1920}
\newcommand{\imageHeight}{1080}

\begin{figure}
\centering
\begin{tikzpicture}
\begin{scope}[scale=10*\pxToMetric*\renderScale]
\coordinate (d1) at (609.500mm, 393.654mm);
\coordinate (d2) at (451.588mm, 30.955mm);
\coordinate (d3) at (839.662mm, 92.118mm);
\coordinate (o1) at (631.054mm, 174.816mm);
\coordinate (ax) at (887.632mm, 543.492mm);
\coordinate (ay) at (1369.466mm, 541.844mm);
\coordinate (az) at (1126.066mm, 941.984mm);
\coordinate (o2) at (1118.542mm, 639.549mm);
\coordinate (r) at (1495.859mm, 221.154mm);
\coordinate (s0) at (608.083mm, 533.215mm);
\coordinate (sl) at (1832.694mm, 552.899mm);
\end{scope}

\tikzset{
  double -latex/.style args={#1 colored by #2 and #3}{
    -latex,line width=#1,#2,
    postaction={draw,-latex,#3,line width=(#1)/3,shorten <=(#1)/4,shorten >=4.5*(#1)/3},
  },
  double round cap-latex/.style args={#1 colored by #2 and #3}{
    round cap-latex,line width=#1,#2,
    postaction={draw,round cap-latex,#3,line width=(#1)/3,shorten <=(#1)/4,shorten >=4.5*(#1)/3},
  },
  double round cap-stealth/.style args={#1 colored by #2 and #3}{
    round cap-stealth,line width=#1,#2,
    postaction={round cap-stealth,draw,,#3,line width=(#1)/3,shorten <=(#1)/3,shorten >=2*(#1)/3},
  },
  double -stealth/.style args={#1 colored by #2 and #3}{
    -stealth,line width=#1,#2,
    postaction={-stealth,draw,,#3,line width=(#1)/3,shorten <=(#1)/3,shorten >=2*(#1)/3},
  },
  double -triangle 45/.style args={#1 colored by #2 and #3}{
    -triangle 45,line width=#1,#2,
    postaction={-triangle 45,draw,,#3,line width=(#1)/3,shorten <=(#1)/3,shorten >=3.5*(#1)/2},
  },
}

\path[use as bounding box] (\horizontalCrop, 0) rectangle (\renderScale*\pxToMetric*\imageWidth - \horizontalCrop, \renderScale*\pxToMetric*\imageHeight);

\node[scale=\outlineScale, anchor=south west, inner sep=0pt] at (0, 0) {\includegraphics{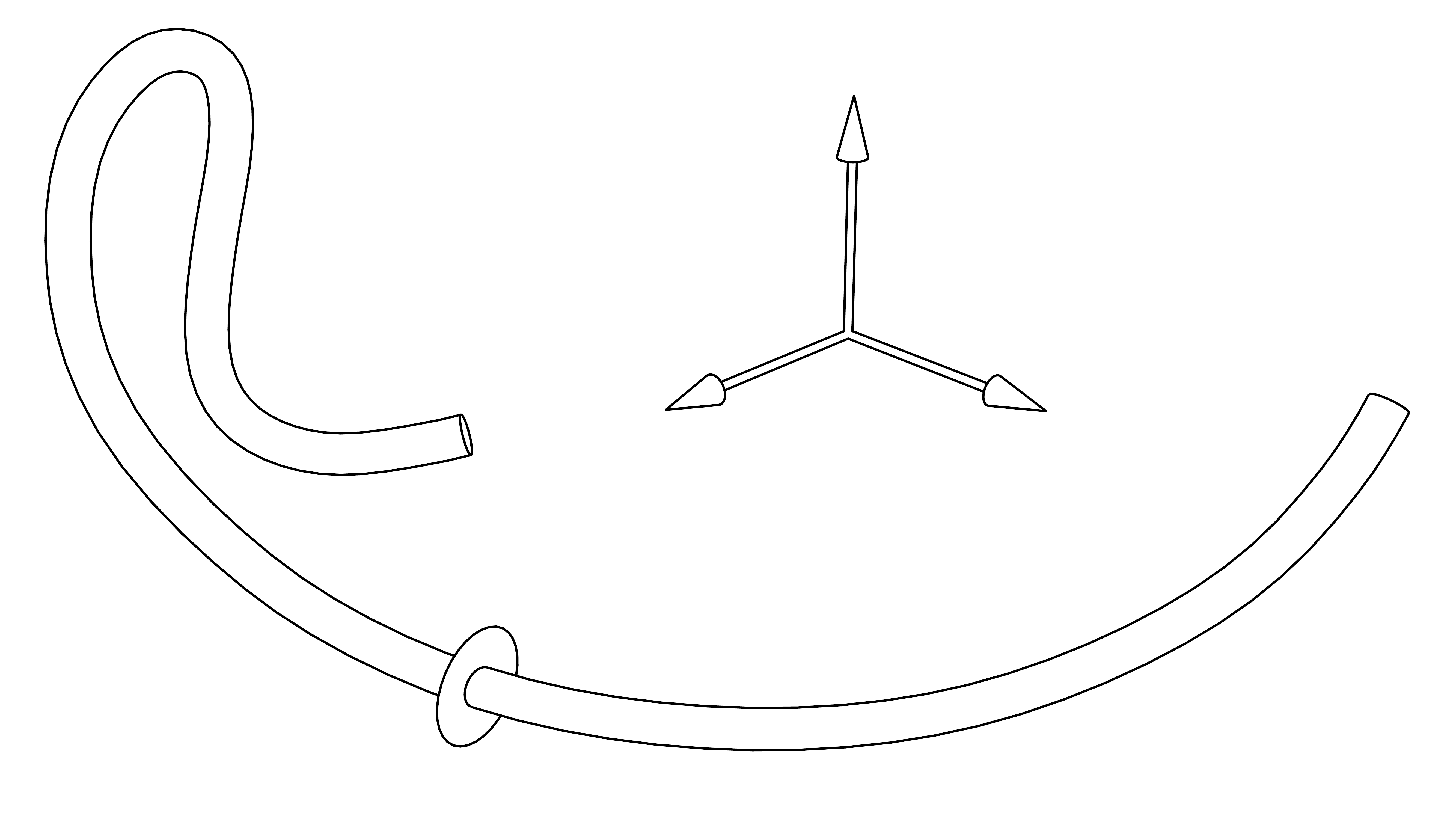}};

\draw[double -latex=2pt colored by black and white] (o2) -- (r) node[black, pos=0.65, right, inner sep=5pt, xshift=0pt, yshift=0pt] {\contour{white}{$\boldsymbol{r}(s, t)$}};

\draw[white, thick, fill=black] (o2) circle (0.1);

\draw[double -latex=2pt colored by black and white] (o1) -- (d1) node[black, anchor=west, inner sep=2pt, xshift=0pt, yshift=0pt] {\contour{white}{$\boldsymbol{d_1}$}};
\draw[double -latex=2pt colored by black and white] (o1) -- (d2) node[black, anchor=south east, inner sep=0pt, xshift=0pt, yshift=0pt] {\contour{white}{$\boldsymbol{d_2}$}};
\draw[double -latex=2pt colored by black and white] (o1) -- (d3) node[black, anchor=west, inner sep=0pt, xshift=0pt, yshift=0pt] {\contour{white}{$\boldsymbol{d_3}$}};

\node [black, anchor=north east, inner sep=2pt, xshift=0pt, yshift=0pt] at (ax) {\contour{white}{$\boldsymbol{x}$}};
\node [black, anchor=north west, inner sep=2pt, xshift=0pt, yshift=0pt] at (ay) {\contour{white}{$\boldsymbol{y}$}};
\node [black, anchor=south, inner sep=4pt, xshift=0pt, yshift=0pt] at (az) {\contour{white}{$\boldsymbol{z}$}};

\node [black, anchor=south, inner sep=3pt, xshift=0pt, yshift=0pt] at (s0) {\contour{white}{$s = 0$}};
\node [black, anchor=south, inner sep=3pt, xshift=0pt, yshift=0pt] at (sl) {\contour{white}{$s = l$}};

\end{tikzpicture}
\caption{The vector set $\{\vec{d}_k\}$ forms a right-handed orthonormal basis at each point of the centerline. The directors $\vec{d}_1$ and $\vec{d}_2$ span the local material cross-section, whereas $\vec{d}_3$ is perpendicular to the material cross-section. Note that in the presence of shear deformations  $\vec{d}_3$ is not tangent to the centerline of the fiber.}
\label{fig:fiberCurveTriad}
\end{figure}
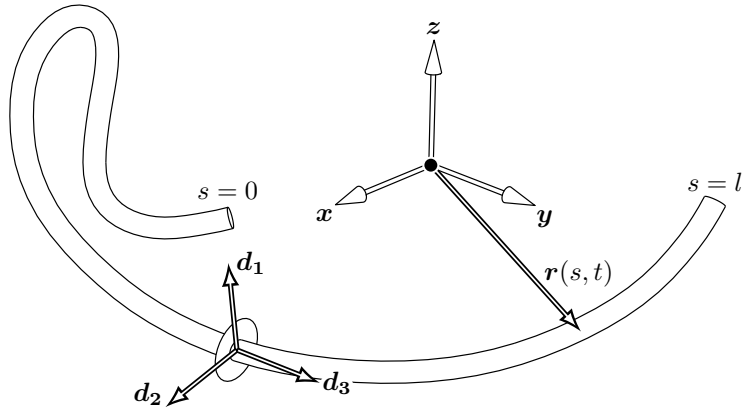

Further, there exist vector-valued functions $\boldsymbol{\kappa}$ and $\boldsymbol{\omega}$ such that the directors evolve according to the kinematic relations
\begin{eqnarray*}
 \partial_s \boldsymbol{d}_k &=& \boldsymbol{\kappa}\times\boldsymbol{d}_k,\\
 \partial_t \boldsymbol{d}_k &=& \boldsymbol{\omega}\times\boldsymbol{d}_k,
\end{eqnarray*}
where $\boldsymbol{\kappa}$ is the Darboux and $\boldsymbol{\omega}$ the twist vector.
Their components are given with respect to the orthonormal basis, i.e.
$\boldsymbol{\kappa}=\sum_{k=1}^3\kappa_k\boldsymbol{d}_k$ and $\boldsymbol{\omega}=\sum_{k=1}^3\omega_k\boldsymbol{d}_k$.

The linear strains of the rod are given by $\boldsymbol{\nu}=\sum_{k=1}^3\nu_k\boldsymbol{d}_k=\partial_s\boldsymbol{r}$ and the velocity of a cross-section material plane by $\boldsymbol{\upsilon}=\partial_t\boldsymbol{r}$. The triples $(\kappa_1, \kappa_2, \kappa_3)$, $(\omega_1, \omega_2, \omega_3)$, $(\nu_1, \nu_2, \nu_3)$, and $(\upsilon_1, \upsilon_2, \upsilon_3)$ are denoted by $\boldsymbol{\mathsf{\kappa}}$, $\boldsymbol{\mathsf{\omega}}$, $\boldsymbol{\mathsf{\nu}}$, and $\boldsymbol{\mathsf{\upsilon}}$ respectively. In particular, $\boldsymbol{\mathsf{\kappa}}:=(\kappa_1, \kappa_2, \kappa_3)$ and $\boldsymbol{\mathsf{\nu}}:=(\nu_1, \nu_2, \nu_3)$ are the strain variables that uniquely determine the motion of the rod described by \refeqn{eqn:ConfigurationCosseratRod} at every instant in time $t$ (except for a rigid body motion). Their components have a physical meaning: they describe the bending of the rod with respect to the two major axes of the cross section ($\kappa_1$,$\kappa_2$), the torsion ($\kappa_3$), shear ($\nu_1$,$\nu_2$) and extension ($\nu_3$). Moreover, since $\partial_t\partial_s\boldsymbol{d}_k=\partial_s\partial_t\boldsymbol{d}_k$ we obtain the compatibility equation
\begin{equation*}
 \partial_t\boldsymbol{\kappa}=\partial_s\boldsymbol{\omega}+\boldsymbol{\omega}\times\boldsymbol{\kappa}.
\end{equation*}
In the same sense we have
\begin{equation*}
 \partial_{t}\boldsymbol{\nu}=\partial_{s}\boldsymbol{\upsilon}.
\end{equation*}

\subsection{Equations of Motion}
The equations of motion for the rod read
\begin{eqnarray*}
 \partial_s\boldsymbol{n}+\boldsymbol{f} &=& \rho A\partial_t{\boldsymbol{\upsilon}}+
  \rho\left(I_1\partial_{tt}\boldsymbol{d}_1+I_2\partial_{tt}\boldsymbol{d}_2\right),\\
 \partial_s\boldsymbol{m}+\boldsymbol{\nu}\times\boldsymbol{n}+\boldsymbol{l} &=&
  \rho\left(I_1\boldsymbol{d}_1+I_2\boldsymbol{d}_2\right)\times\partial_t{\boldsymbol{\upsilon}}+
   \partial_t\left(\rho\boldsymbol{J}\boldsymbol{\omega}\right),
\end{eqnarray*}
where $\boldsymbol{n}=\sum_{k=1}^3n_k\boldsymbol{d}_k$ and $\boldsymbol{m}=\sum_{k=1}^3m_k\boldsymbol{d}_k$ are the internal stresses and $\boldsymbol{f}$ and $\boldsymbol{l}$ are the external forces and torques acting on the rod, $\varrho$ the linear density. $I_1$ and $I_2$ are the first mass moments of inertia of cross section per unit length and $\boldsymbol{J}$ is the inertia tensor of cross section per unit length. Further, we define $\boldsymbol{\mathsf{n}}:=(n_1,n_2,n_3)$ and $\boldsymbol{\mathsf{m}}:=(m_1,m_2,m_3)$. The shear forces are given by $n_1$ and $n_2$, the tension by $\boldsymbol{n}\cdot\boldsymbol{\nu}/\left\Vert\boldsymbol{\nu}\right\Vert$, bending moments by $m_1$ and $m_2$, and the twisting moment by $m_3$.

\subsection{Constitutive Relations}
In order to relate the internal stresses $\boldsymbol{n}$ and $\boldsymbol{m}$ to the kinematic quantities $\boldsymbol{\nu}$ and $\boldsymbol{\kappa}$ we introduce constitutive equations of the form
\begin{eqnarray*}
  \boldsymbol{\mathsf{n}}(s,t)&=&\hat{\boldsymbol{\mathsf{n}}}\left(
   \boldsymbol{\mathsf{\kappa}}(s, t),\boldsymbol{\mathsf{\nu}}(s,t), s
    \right),\\
  \boldsymbol{\mathsf{m}}(s,t)&=&\hat{\boldsymbol{\mathsf{m}}}\left(
   \boldsymbol{\mathsf{\kappa}}(s, t),\boldsymbol{\mathsf{\nu}}(s,t), s
    \right).
\end{eqnarray*}
For fixed $s$, the common domain $\mathcal{V}(s)$ of these constitutive functions is a subset of $(\boldsymbol{\mathsf{\kappa}},\boldsymbol{\mathsf{\nu}})$ describing orientation preserving deformations. $\mathcal{V}(s)$ consists at least of all $(\boldsymbol{\mathsf{\kappa}},\boldsymbol{\mathsf{\nu}})$ that satisfy $\nu_3=\boldsymbol{\nu}\cdot\boldsymbol{d}_3>0$.

The rod is called hyper-elastic, if there exists a strain-energy-function $W:\{(\boldsymbol{\mathsf{\kappa}},\boldsymbol{\mathsf{\nu}}\in\mathcal{V})\}\rightarrow\mathbb{R}$ such that
\begin{eqnarray*}
 \hat{\boldsymbol{\mathsf{n}}}\left(
   \boldsymbol{\mathsf{\kappa}},\boldsymbol{\mathsf{\nu}}, s
    \right) &=& \partial W\left(
   \boldsymbol{\mathsf{\kappa}},\boldsymbol{\mathsf{\nu}}, s
    \right)/\partial\boldsymbol{\mathsf{\nu}},\\
 \hat{\boldsymbol{\mathsf{m}}}\left(
   \boldsymbol{\mathsf{\kappa}},\boldsymbol{\mathsf{\nu}}, s
    \right) &=& \partial W\left(
   \boldsymbol{\mathsf{\kappa}},\boldsymbol{\mathsf{\nu}}, s
    \right)/\partial\boldsymbol{\mathsf{\kappa}}.
\end{eqnarray*}

The rod is called viscoelastic of strain-rate type of complexity 1 if there exist functions such that
\begin{eqnarray}
 \label{eqn:ConstitutiveViscoElastic1}
  \boldsymbol{\mathsf{n}}(s,t)&=&\hat{\boldsymbol{\mathsf{n}}}\left(
   \boldsymbol{\mathsf{\kappa}}(s, t),\boldsymbol{\mathsf{\nu}}(s,t),
   \partial_t\boldsymbol{\mathsf{\kappa}}(s, t),\partial_t\boldsymbol{\mathsf{\nu}}(s,t),s
    \right),\\
  \label{eqn:ConstitutiveViscoElastic2}
  \boldsymbol{\mathsf{m}}(s,t)&=&\hat{\boldsymbol{\mathsf{m}}}\left(
   \boldsymbol{\mathsf{\kappa}}(s, t),\boldsymbol{\mathsf{\nu}}(s,t),
   \partial_t\boldsymbol{\mathsf{\kappa}}(s, t),\partial_t\boldsymbol{\mathsf{\nu}}(s,t),s
    \right).
\end{eqnarray}
 For $\partial_t\boldsymbol{\mathsf{\kappa}}(s, t)=\boldsymbol{\mathsf{0}}$ and $\partial_t\boldsymbol{\mathsf{\nu}}(s, t)=\boldsymbol{\mathsf{0}}$, \refeqn{eqn:ConstitutiveViscoElastic1} and \refeqn{eqn:ConstitutiveViscoElastic2} become the so called equilibrium response functions and describe elastic behavior.

\subsection{Material Laws}
The constitutive laws for elastic material behavior become
\begin{equation*}
\hat{\boldsymbol{\mathsf{n}}}(s,t)=\left(G A \left(\nu_1-\nu_1^\circ \right), G A \left(\nu_2-\nu_2^\circ \right), E A \left(\nu_3-\nu_3^\circ \right)\right),
\end{equation*}
with the initial strain vector field $\vec{\nu}^\circ(s)$, Young's modulus $E$, cross-section area $A$, and
\begin{equation*}
\hat{\boldsymbol{\mathsf{m}}}(s,t)=\left(E_b I_1 \left(\kappa_1-\kappa_1^\circ \right), E_b I_2 \left(\kappa_2-\kappa_2^\circ \right), G I_3 \left(\kappa_3-\kappa_3^\circ \right)\right),
\end{equation*}
with the initial bending and torsion vector field $\vec{\kappa}^\circ(s)$, Young's modulus $E_b$ of bending, and shear modulus $G$. The area moments of inertia are again denoted by $I_1$ and $I_2$, the polar moment of inertia with $I_3$.

Since bending stiffnesses $E_bI_1$, $E_bI_2$, and torsional stiffness $G_bI_3$ change with the fourth power of the fiber diameter they are usually orders of magnitude smaller than the tensile stiffness $EA$ and shearing stiffness $GA$. This renders the problem of fiber simulation based on the Special Theory of Cosserat Rods inherently ``stiff''.

\subsubsection{Kirchhoff Rods} We allude to the fact that in the classical theory of Kirchhoff the rod can undergo neither shear nor extension. This is accommodated by setting the linear strains to $\boldsymbol{\mathsf{\nu}}:=(\nu_1,\nu_2,\nu_3)=(0, 0, 1)$, (local coordinates). Geometrically this means, that the angle between the director $\boldsymbol{d}_3$ and the tangent to the centerline, $\partial_s\boldsymbol{r}$, always remains zero (no shear) and that the tangent to the centerline always has unit length (no elongation).

\subsection{System of Governing Equations} The full system of partial differential equations governing the deformation of an elastic rod is thus given by the following first order system,
\begin{eqnarray}
  \label{eqn:KinematicEquation}
  \partial_t\boldsymbol{d}_k &=&
   \boldsymbol{\omega}\times\boldsymbol{d}_k,\\
  \label{eqn:cosserat1}
  \partial_t\boldsymbol{\kappa} &=&
   \partial_s\boldsymbol{\omega}+\boldsymbol{\omega}\times\boldsymbol{\kappa},\\
\nonumber
  \partial_t\boldsymbol{\nu} &=& \partial_s\boldsymbol{\upsilon},\\
\nonumber
  \partial_t\left(\rho\boldsymbol{J}\boldsymbol{\omega}\right) &=&
   \partial_s\left(\hat{m}_k(\boldsymbol{\mathsf{\kappa}},\boldsymbol{\mathsf{\nu}})\boldsymbol{d}_k\right) +
    \boldsymbol{\nu}\times\hat{n}_k(\boldsymbol{\mathsf{\kappa}},\boldsymbol{\mathsf{\nu}})\boldsymbol{d}_k,\\
\nonumber
  \rho A\partial_t\boldsymbol{\upsilon} &=&
   \partial_s\left(\hat{n}_k(\boldsymbol{\mathsf{\kappa}},\boldsymbol{\mathsf{\nu}})\boldsymbol{d}_k\right).
\end{eqnarray}

If $(\hat{\boldsymbol{\mathsf{n}}},\hat{\boldsymbol{\mathsf{m}}})$ satisfy the monotonicity condition, i.e. the matrix
\begin{equation*}
  \begin{bmatrix}
    \partial\hat{\boldsymbol{\mathsf{m}}}/\partial\boldsymbol{\mathsf{\kappa}} &
    \partial\hat{\boldsymbol{\mathsf{m}}}/\partial\boldsymbol{\mathsf{\nu}} \\
    \partial\hat{\boldsymbol{\mathsf{n}}}/\partial\boldsymbol{\mathsf{\kappa}} &
    \partial\hat{\boldsymbol{\mathsf{n}}}/\partial\boldsymbol{\mathsf{\nu}}
  \end{bmatrix},
\end{equation*}
is positive-definite, then this system is hyperbolic. It can be written in the form of a conservation law
\begin{equation*}
 \partial_t\Phi(\zeta)=\partial_s\Psi(\zeta)+\Theta(\zeta),
\end{equation*}
with $\zeta=(\boldsymbol{d}_k,\boldsymbol{\kappa},\boldsymbol{\nu},\boldsymbol{\omega}, \boldsymbol{\upsilon})$. This system can be decoupled from the Kinematic \refeqn{eqn:KinematicEquation} by decomposing it with respect to the basis $\{\boldsymbol{d}_k\}$ which yields
\begin{eqnarray}
\label{eqn:cosserat1b}
  \partial_t\boldsymbol{\mathsf{\kappa}} &=&
   \partial_s\boldsymbol{\mathsf{\omega}}+\boldsymbol{\mathsf{\omega}}\times\boldsymbol{\mathsf{\kappa}}, \label{maineq}\\
\nonumber
  \partial_t\boldsymbol{\mathsf{\nu}} &=& \partial_s\boldsymbol{\mathsf{\upsilon}}+
    \boldsymbol{\mathsf{\kappa}}\times\boldsymbol{\mathsf{\upsilon}} -
    \boldsymbol{\mathsf{\omega}}\times\boldsymbol{\mathsf{\nu}},\\
\nonumber
  \partial_t\left(\rho\boldsymbol{\mathsf{J}}\boldsymbol{\mathsf{\omega}}\right) &=&
   \partial_s\hat{\boldsymbol{\mathsf{m}}} +
   \boldsymbol{\mathsf{\kappa}}\times\hat{\boldsymbol{\mathsf{m}}} +
   \boldsymbol{\mathsf{\nu}}\times\hat{\boldsymbol{\mathsf{n}}} -
   \boldsymbol{\mathsf{\omega}}\times\left(\rho\boldsymbol{\mathsf{J}}\boldsymbol{\mathsf{\omega}}\right),\\
\label{eqn:LinearVelocity}
  \rho A\partial_t\boldsymbol{\mathsf{\upsilon}} &=&
   \partial_s\hat{\boldsymbol{\mathsf{n}}} +
   \boldsymbol{\mathsf{\kappa}}\times\hat{\boldsymbol{\mathsf{n}}} -
   \boldsymbol{\mathsf{\omega}}\times\left(\rho A\boldsymbol{\mathsf{\upsilon}}\right).
\end{eqnarray}
If external forces (e.g. gravity) are to be considered as well, they have to be added to the right-hand side of \refeqn{eqn:LinearVelocity} after transforming them into the local basis. For this purpose the Kinematic \refeqn{eqn:KinematicEquation} has to be solved additionally.

\section{Lie Symmetry Analysis}

\subsection{Problem Setting}
As a first step in the direction of a better understanding of the Cosserat rod we begin our study with a subsystem of the Special Cosserat Theory of Rods. Precisely, we apply the classical symmetry method invented by Sophus Lie (cf.~\cite{Olivery10}) to the equation system given by \refeqn{maineq},
\begin{equation}\label{veq}
\vec{F}=\vec{0},\quad \vec{F}:=\partial_s \vec{\omega} - \partial_t \vec{\kappa} +  \vec{\omega} \times \vec{\kappa},
\end{equation}
which corresponds to \refeqn{eqn:cosserat1} or \refeqn{eqn:cosserat1b}.

\subsection{Infinitesimal Criterion of Invariance}
We apply the classical Lie-point symmetry method~\cite{Olivery10} to the quasi-linear first-order PDE system \refeqn{veq} in two independent variables $s,t$ and six dependent variables that are the components of the vectors $\vec{\omega}$ and $\vec{\kappa}$.

A transformation of \refeqn{veq}
\begin{equation}
\begin{array}{l}
s'=s'(s,t,\vec{\omega}(s,t),\vec{\kappa}(s,t)),\quad \vec{\omega}'=\vec{\omega}'(s,t,\vec{\omega}(s,t),\vec{\kappa}(s,t)), \\[0.1cm] t'=t'(s,t,\vec{\omega}(s,t),\vec{\kappa}(s,t)),\,
\quad \vec{\kappa}'=\vec{\kappa}'(s,t,\vec{\omega}(s,t),\vec{\kappa}(s,t)),
\end{array}
\label{group}
\end{equation}
that maps solutions to solutions is called Lie(-point) symmetry of \refeqn{veq}. The vector field
\begin{equation}
X:=\xi^{1}{\partial_s}+\xi^{2}{\partial_t}+\sum_{i=1}^3\left(\theta^{i}{\partial_{ \omega_i}}+\vartheta^{i}{\partial_{\kappa_i}}\right), \label{generator}
\end{equation}
is an infinitesimal generator of a one-parameter $(x\in \mathbb{R})$ Lie group of point transformations if its flow $\exp(x\,X)$  is a Lie(-point) symmetry. The coefficients $\xi^{1},\xi^{2},\theta^{i},\vartheta^{j}$ $(i,j=1,2,3)$ in \refeqn{generator} are functions in independent and dependent variables, and below we shall use the vector notation $\vec{\theta}:=\{\theta^1,\theta^2,\theta^3\}$, $\vec{\vartheta}:=\{\vartheta^1,\vartheta^2,\vartheta^3\}$.

The equality
\begin{equation}
X^{(pr)}\vec{F}\mid_{\vec{F}=0}=0, \label{criterion}
\end{equation}
is the {\em infinitesimal criterion of invariance} of \refeqn{veq} under a one-parameter Lie group of point transformations. Here $X^{(pr)}$ stands for the prolonged infinitesimal symmetry generator that, in addition to those in \refeqn{generator}, contains extra terms caused by the presence of the first-order partial derivatives in \refeqn{veq}. These extra terms are easily computed taking certain derivatives of the coefficient functions in the generator. The subscript in \refeqn{criterion} indicates that equality $X^{(pr)}\vec{F}=0$ must hold under condition $\vec{F}=0$. Equality \refeqn{criterion} implies an overdetermined determining system of linear PDEs for the coefficient functions of generator \refeqn{generator}. Since the equations in \refeqn{veq} are quasi-linear and solved with respect to partial derivatives, generation of their determining equations is algorithmically straightforward.

\subsection{Solving Determining Equations}
To obtain the determining equations we use the Maple package {\sc Desolv}~\cite{CarminatiVu00} and its routine {\em gendef}. It outputs 42 first-order PDEs. Generally, a reliable and powerful algorithmic way to solve a system  determining equations is its  transformation to a canonical involutive form or to a Gr\"{o}bner basis form and then solving such canonical system (cf.~\cite{Hereman96}). The package {\sc Desolv} has the built-in routine {\em icde} that implements the Standard Form algorithm~\cite{Reid91} for completion to involution. We prefer, however, completion to a Janet basis, the canonical involutive form based on Janet division (cf.~\cite{Seiler10} for theory of involution for algebraic and differential systems and references therein) and respectively, the Maple package {\sc Janet}~\cite{BCGPR03-2} that computes a Janet basis for a differential ideal generated by linear differential polynomials. We have two arguments in favor of this preference: (i) according to our benchmarking, package {\sc Janet} is substantially faster than the routine {\em icde} in {\sc Desolv}; (ii) given a linear system of PDEs in the Janet involutive form and a set of its analytic solutions one can algorithmically check whether this set contains all analytic solutions (cf.~\cite{Langer-Hegermann2014,Seiler10}).

The involutive form of determining equations for \refeqn{criterion} computed by package {\sc Janet} contains 86 linear PDEs.
To solve these equations we applied the routine {\em pdsolv} of the package {\sc Desolv} that exploits a number of heuristic algorithms for integration of linear PDEs including some advanced algorithms~\cite{Wolf00} oriented to integration of overdetermined systems of polynomially-nonlinear differential equations. It outputs a solution that depends on five arbitrary functions in the independent variables $s,t$. Two of these functions are unnecessary for us and can be omitted since they appear in the solution as shifts in the independent variables in accordance to the fact that the equations in \refeqn{veq} are autonomous. Taking this into account, the obtained solution can be presented as
\begin{equation}
\xi^1=\xi^2=0,\quad \vec{\theta}=\boldsymbol{\hat{A}}\vec{\omega}+{\partial_t \vec{p}},\quad \vec{\vartheta}=\boldsymbol{\hat{A}}\vec{\kappa}+{\partial_s \vec{p}}, \label{solution}
\end{equation}
where
\begin{equation}
\boldsymbol{\hat{A}}(s,t)=\begin{bmatrix} \phantom{-}0 & -c(s,t) & \phantom{-}b(s,t) \\ \phantom{-}c(s,t) & \phantom{-}0 & -a(s,t) \\ -b(s,t) & \phantom{-}a(s,t) & \phantom{-}0 \end{bmatrix},\quad
\vec{p}(s,t):=\begin{bmatrix} a(s,t) \\ b(s,t) \\ c(s,t) \end{bmatrix}, \label{matrix}
\end{equation}
and $a(s,t),b(s,t),c(s,t)$ are arbitrary smooth functions. Hereafter, we shall assume analyticity of these functions.

\subsection{Lie Group of Point Symmetry Transformations}
The routine {\em pdesolv} of {\sc Desolv} may not always find all solutions of the input differential system. For our purpose, however, the arbitrariness given by \refeqn{solution} and \refeqn{matrix} is sufficient for the construction of a general solution to \refeqn{veq}. It is essential  that the routine outputs a solution with maximally possible arbitrariness in the number of arbitrary functions depending on the both independent variables $(s,t)$. This fact can be detected from the structure of differential dimensional polynomial \cite{Langer-Hegermann2014}. It is easily computed by the routine {\em DifferentialSystemDimensionPolynomial} of the Maple package {\sc DifferentialThomas}~\cite{BGLHR12} that takes the Janet basis form of the determining system as an input. In the case being considered the differential dimensional polynomial is given by
\[
\frac{5}{2}l^2+\frac{21}{2}l+11=5\binom{l+2}{l}+3\binom{l+1}{l}+3.
\]
The first term of this expression shows that the general analytic solution contains 5 arbitrary functions in 2 variables. If one applies the built-in Maple command {\em pdsolve}, then its output solution set also contains 5 arbitrary functions in two variables. However, this solution depends nonlocally (via some integrals) on the arbitrary functions, and is rather cumbersome.

Having obtained the structure \refeqn{solution} and \refeqn{matrix} of the infinitesimal symmetry generator \refeqn{generator}, we compute now the one-parameter Lie-point symmetry group of transformations \refeqn{group} it generates. The symmetry group is given as solution (Lie's first fundamental theorem~\cite{Olivery10}) of the following system of two trivially solvable scalar ordinary  differential equations
\begin{equation*}
\begin{array}{l}
\mathrm{d}_xs'=0,\quad s'(0)=s\quad \Longrightarrow \quad s'=s,\\[0.1cm]
\mathrm{d}_xt'=0,\,\quad t'(0)=t\,\quad \Longrightarrow \quad t'=t,
\end{array}
\end{equation*}
and two vector ones
\begin{align}
&\mathrm{d}_x\vec{\omega}'=\boldsymbol{\hat{A}}\vec{\omega}'+{\partial_t \vec{p}}\,,\quad \vec{\omega}'(0)=\vec{\omega}, \label{Lie_omega} \\[0.1cm]
&\mathrm{d}_x\vec{\kappa}'=\boldsymbol{\hat{A}}\vec{\kappa}'+{\partial_s \vec{p}}\,,\,\quad \vec{\kappa}'(0)=\vec{\kappa}. \label{Lie_kappa}
\end{align}
Solution to \refeqn{Lie_omega} and \refeqn{Lie_kappa} is computable with the Maple command {\em dsolve}. But the output of this command is awkward and it is not easy to obtain in compact form. Such compact form solution can be obtained by hand computation as follows.

It is not difficult to see (cf.~\cite{Bellman97}), that \refeqn{Lie_omega} is satisfied by the expression
\begin{equation}
\vec{\omega}'=\exp(x\boldsymbol{\hat{A}})\vec{\omega}+\exp(x\boldsymbol{\hat{A}})\int_0^x\exp(-y\boldsymbol{\hat{A}})\,dy\,\,{\partial_t \vec{p}}\,. \label{group_omega}
\end{equation}
Since this system satisfies the conditions of the classical existence and uniqueness theorem for systems of ordinary differential equations~\cite{Pontryagin62},  it follows that \refeqn{group_omega} is the solution to \refeqn{Lie_omega}.

Furthermore, application of the Cayley-Hamilton theorem~\cite{AtiyahMacdonald69} to the matrix $\boldsymbol{\hat{A}}$ in \refeqn{matrix} gives
\begin{equation}
\boldsymbol{\hat{A}}^3=-p^2\boldsymbol{\hat{A}},\qquad p:=|\vec{p}|=\sqrt{a^2(s,t)+b^2(s,t)+c^2(s,t)}. \label{CH}
\end{equation}

By means of relation \refeqn{CH} expression \refeqn{group_omega} is transformed to
\begin{multline}
\vec{\omega}' = \left(\boldsymbol{\hat{I}}+\frac {\sin \left( p\, x
 \right)\boldsymbol{\hat{A}}}{p} + {\frac {\left( 1-\cos \left( p\, x \right)
 \right) \boldsymbol{\hat{A}}^{2}}{{p}^{2}}} \right)\vec{\omega} + \\ + \left({\frac {p\, x\boldsymbol{\hat{A}}^{2}+p^3\, x\boldsymbol{\hat{I}}
 -\sin \left( p\, x \right)\boldsymbol{\hat{A}}^{2} - \cos \left( p\,x \right)p\boldsymbol{\boldsymbol{\hat{A}}} + p\boldsymbol{\hat{A}}
 }{{p}^{3}}}\right) {\partial_t \vec{p}}\,, \label{matsol}
\end{multline}
where $\boldsymbol{\hat{I}}$ is the $3\times 3$ identity matrix. Formulae \refeqn{matrix} and \refeqn{matsol} show that without loss of generality the arbitrary vector $\vec{p}$ and matrix $\boldsymbol{\hat{A}}$ for $x\neq 0$ can be rescaled to absorb the group parameter $x$. It is equivalent to putting $x:=1$. In so doing, transformation \refeqn{matsol} can be rewritten in terms of arbitrary vector-function $\vec{p}$ as follows,
\begin{multline}
\vec{\omega}' = \vec{\omega} - \frac{\sin(p)}{p}\,\vec{p}\times \vec{\omega} + \frac{1-\cos(p)}{p^2}\,\left(\vec{p}\,(\vec{p}\, \vec{\omega})-p^2\,\vec{\omega}\right)+{\partial_t \vec{p}}+\\
  +\frac{p-\sin(p)}{p^3}\,\left(\vec{p}\,\left(\vec{p}\, {\partial_t \vec{p}}\right)-p^2\,{\partial_t \vec{p}}\right)-\frac{1-\cos(p)}{p^2}\,\vec{p}\times {\partial_t \vec{p}}\,. \label{vecsol_omega}
\end{multline}

Respectively, solution to \refeqn{Lie_kappa} reads
\begin{multline}
\vec{\kappa}' = \vec{\kappa} - \frac{\sin(p)}{p}\,\vec{p}\times \vec{\kappa} + \frac{1-\cos(p)}{p^2}\,\left(\vec{p}\,(\vec{p}\, \vec{\kappa})-p^2\,\vec{\kappa}\right)+{\partial_s \vec{p}}+\\
  +\frac{p-\sin(p)}{p^3}\,\left(\vec{p}\,\left(\vec{p}\, {\partial_s \vec{p}}\right)-p^2\,{\partial_s \vec{p}}\right)-\frac{1-\cos(p)}{p^2}\,\vec{p}\times {\partial_s \vec{p}}\,. \label{vecsol_kappa}
\end{multline}

Thus, transformations \refeqn{vecsol_omega} and \refeqn{vecsol_kappa} are Lie-point symmetries of \refeqn{veq} for an arbitrary vector $\vec{p}$ in \refeqn{matrix}. We also verified directly with Maple that if $\vec{\omega},\vec{\kappa}$ are solutions to \refeqn{veq}, then $\vec{\omega}',\vec{\kappa}'$ are also solutions for any vector-function $\vec{p}$.

\subsection{General Solution}
Consider now the special case of transformations \refeqn{vecsol_omega} and \refeqn{vecsol_kappa} when
\begin{equation}
\vec{\omega}:=\vec{f}(t),\quad \vec{f}(t)=\{f_1(t),f_2(t),f_3(t)\},\quad \vec{\kappa}:=0, \label{init_data}
\end{equation}
where $f_1(t),f_2(t),f_3(t)$ are arbitrary analytic functions, and denote the image of \refeqn{init_data} under the transformations by
$\vec{\omega}$ and $\vec{\kappa}$:
\begin{equation}\label{gen_sol}
\begin{split}
\vec{\omega} =& \vec{f}(t) - \frac{\sin(p)}{p}\,\vec{p}\times \vec{f}(t) + \frac{1-\cos(p)}{p^2}\,\left(\vec{p}\,(\vec{p}\, \vec{f}(t))-p^2\,\vec{f}(t)\right)+{\partial_t \vec{p}}+\\
  & +\frac{p-\sin(p)}{p^3}\,\left(\vec{p}\,\left(\vec{p}\, {\partial_t \vec{p}}\right)-p^2\,{\partial_t \vec{p}}\right)-\frac{1-\cos(p)}{p^2}\,\vec{p}\times {\partial_t \vec{p}}\,,\\
\vec{\kappa} = & {\partial_s \vec{p}}
  +\frac{p-\sin(p)}{p^3}\,\left(\vec{p}\,\left(\vec{p}\,{\partial_s \vec{p}}\right)-p^2\,{\partial_s \vec{p}}\right)-\frac{1-\cos(p)}{p^2}\,\vec{p}\times {\partial_s \vec{p}}\,.
\end{split}
\end{equation}

Now we can formulate and prove the main theoretical result of the present paper.
\begin{proposition}
The vector-functions $\vec{\omega}$ and $\vec{\kappa}$ expressed by formulae \refeqn{gen_sol} in terms of the matrix- and vector-functions \refeqn{matrix}, \refeqn{init_data} whose components are arbitrary analytic functions provide a solution to equations \refeqn{veq} that is general.
\end{proposition}

\begin{proof}
Obviously, \refeqn{init_data} is a solution. It follows that expressions \refeqn{gen_sol} make up a solution. We show first that one can choose arbitrary functions $a(s,t)$, $b(s,t)$, and $c(s,t)$ to obtain any vector-function $\vec{\kappa}$ as the left-hand side of the second equality in \refeqn{gen_sol}. The partial derivatives $\partial_s a$, $\partial_s b$, $\partial_s c$ of $a,b,c$ with respect to $s$ appear linearly in \refeqn{gen_sol}. Direct computation with Maple of the Jacobian matrix $J_{\vec{\kappa}}(\partial_s a,\partial_s b, \partial_s c)$ and its determinant gives the following compact expression:
\begin{equation*}
   \det\left(J_{\vec{\kappa}}(\partial_s a,\partial_s b, \partial_s
   c)\right)=\frac
   {2(\cos(p)-1)}{p^2}\,.
\end{equation*}

It is clear that one can always choose the {\em initial values} $a(0,t)$, $b(0,t)$, and $c(0,t)$ of the arbitrary functions to provide nonzero values of the Jacobian determinant, and hence to solve locally the last vector equality in \refeqn{gen_sol} with respect to $\partial_s a$, $\partial_s b$, $\partial_s c$. Thereby, the equation can be brought into a first-order differential system solved with respect to the partial derivative and with analytic right-hand sides. Then, by the classical Cauchy-Kovalevskaya theorem (cf.~\cite{Seiler10}) the initial value (Cauchy) problem for the solved system has a unique analytic solution.

We now turn our attention to the first equation in \refeqn{gen_sol} and show that having decided upon $a(s,t)$, $b(s,t)$, and $c(s,t)$, one can choose functions $f_1(t)$, $f_2(t)$, and $f_3(t)$ to obtain arbitrary analytic $\omega_1(0,t)$, $\omega_2(0,t)$, and $\omega_3(0,t)$. The part of \refeqn{gen_sol} linear in $\vec{f}(t)$ is identically equal to $\exp({\boldsymbol{\hat{A}}})\vec{f}(t)$ (cf.~\refeqn{group_omega}). Since the matrix $\boldsymbol{\hat{A}}$ given by \refeqn{matrix} is skew-symmetric, $\det(\exp({\boldsymbol{\hat{A}}}))=1$ and \refeqn{gen_sol} is solvable with respect to the analytic vector-function $\vec{f}(t)$. It implies that the system \refeqn{veq}, being in {\em normal} or {\em Cauchy-Kovalevskaya form}, admits unique analytic solutions for $\omega_1(s,t)$, $\omega_2(s,t)$, and $\omega_3(s,t)$. $\Box$
\end{proof}


\section{Conclusion and Future Work}

In this contribution, we have studied a subsystem of the Special Cosserat Theory of Rods by performing a Lie group analysis. To be more specific, by applying modern computer algebra methods, algorithms, and software we have constructed an explicit form of solution to \refeqn{veq} that depends on three arbitrary functions in $(s,t)$ and three arbitrary functions in $t$. Assuming analyticity of the arbitrary functions in a domain under consideration, we have proved that the obtained solution is analytic and general. This is a step towards generating detailed knowledge about the structure of the PDE system that governs the spatiotemporal evolution of the Cosserat rod.

In our future work these results will be used to develop algorithms based on combinations of numerical and analytical treatments of the governing equations to overcome the typical problems resulting from the system's stiffness. This approach would allow for larger step sizes compared to pure numerical solvers and at the same time combines efficiency and accuracy without sacrificing one for another.

\section*{Acknowledgements}

The contribution of the third author (V.P.G.) was partially supported by the grant 13-01-00668 from the Russian Foundation for Basic Research. We thank Markus Lange--Hegermann and Paul Mueller for useful remarks. The authors are grateful to the CASC 2014 reviewers' valuable comments that improved the
manuscript.



\end{document}